\documentclass[12pt, 14paper,reqno]{amsart}
\usepackage{amsmath,amsfonts,amssymb}
\usepackage{mathrsfs}
\usepackage{longtable}
\usepackage[breaklinks]{hyperref}
\usepackage{graphicx}
\usepackage{mathtools}
\DeclarePairedDelimiter{\ceil}{\lceil}{\rceil}

\theoremstyle{plain}
\newtheorem{thm}{Theorem}[section]
\newtheorem{lem}{Lemma}[section]
\newtheorem{cor}{Corollary}[section]

\newtheorem{thma}{Theorem}

\theoremstyle{proof}

\numberwithin{equation}{section}

\begin{document} 
\title[Sums of integral squares]{Sums of integral squares in complex bi-quadratic fields and in CM fields}
\author{Srijonee Shabnam Chaudhury}
\address{ Srijonee Shabnam Chaudhury
@Harish-Chandra Research Institute,HBNI,
Chhatnag Road, Jhunsi,  Allahabad 211 019, India.}
\email{srijoneeshabnam@hri.res.in}
\maketitle
\begin{abstract}
Let $K$ be a complex bi-quadratic field with ring of integers $\mathcal{O}_{K}$. For $K =  \mathbb{Q}(\sqrt{-m}$, $\sqrt{n}$), where $ m \equiv 3 \pmod 4 $ and $ n \equiv 1 \pmod 4$, we prove that every algebraic integer can be written as sum of integral squares. Using this, we prove that for any complex bi-quadratic field $K$, every element of $4\mathcal{O}_K$ can be written as sum of five integral squares. In addition, we show that the Pythagoras number of ring of integers of any CM field is at most five. Moreover, we give two classes of complex bi-quadratic fields for which $p(\mathcal{O}_{K})= 3$ and $p(4\mathcal{O}_{K})=3 $ respectively. Here, $p(\mathcal{O}_{K})$ is the Pythagoras number of ring of integers of $K$ and $p(4\mathcal{O}_{K})$ is the smallest positive integer $t$ such that every element of $4\mathcal{O}_{K}$ can be written as sum of $t$ integral squares.
\end{abstract}
\maketitle
\section{Introduction}
The sums of integral squares in a number field is one of the fundamental object of study in number theory. In 1770, E.Waring proposed the famous 'Waring's problem' which asks whether each natural number $k$ has an associated positive integer $s$ such that every natural number is the sum of at most $s$ natural numbers to the power of $k$. This problem is a generalisation of Lagrange's four square theorem (see \cite{LA}). Siegel \cite{SI21} studied the Waring's problem over algebraic number fields and showed  in \cite{SI45} that  $\mathbb{Q}$ and $\mathbb{Q}(\sqrt{ 5} )$ are only two totally real fields in which every totally positive integer can be written as sums of integral squares. On the other hand, if $K$ is not totally real then all totally positive algebraic integers are sums of integral squares in $K$ if and only if the discriminant of $K$ is odd. Estes and Hsia \cite{EH} determined all complex quadratic fields in which algebraic integers are expressible as sums of three integer squares. This article is motivated by a recent work of Zhang and Ji  \cite{ZJ} on sums of three squares in certain complex bi-quadratic fields, which we will state later in section 3. Previously we have proved that for some  complex bi-quadratic fields every algebraic integer can be written as sum of integral squares. Now using the result, in this article  we get some generalised results for every complex bi-quadratic field and will prove that every algebraic integer multiplied by $4$ can be written as sum of integral squares. 
For proving these results, we will use some techniques and results from elementary number theory and from local class field theory (only in proof of Corollary \ref{cor2}). The most important fact which helps us to prove these results is that $-1$ can be written as sum of integral squares in complex bi-quadratic fields. This is also true for cyclotomic fields. Therefore, using the fact, similar results can also be proved for any non-totally real number fields. But for totally real cases this is not applicable.
 
\section{Notation}
Let $K$ be any complex bi-quadratic field with ring of integers $ \mathcal{O}_K$ and let $F$ be any number field.
We now define the following:
\begin{itemize}
\item[] $s\mathcal{O}_{K}$ : the set of all elements $\alpha \in K $ such that $ \alpha = s \beta $ where $\beta \in \mathcal{O}_{K}$ and $s \in \mathbb{N}$
\item[] $\mathcal{R}_K$: the set of all elements in $\mathcal{O}_K$ which are sum of squares in $\mathcal{O}_K$. 
\item[] $s(\mathcal{O}_K)$: the minimal number of squares required to represent $-1$ in $\mathcal{O}_K$.
\item[] $p(\mathcal{O}_K)$:  the smallest positive integer $t$ such that every element in $R_K$ is a sum of $t$ squares in $\mathcal{O}_K$. This is called the Pythagoras number of ring of integers of $K$.
\item[] $p(s\mathcal{O}_K)$:  the smallest positive integer $t$ such that every element in $s\mathcal{O}_K$ is a sum of $t$ squares in $\mathcal{O}_K$. We can say this the Pythagoras number of $s\mathcal{O}_{K}$ for the field $K$.
\item[] $D$ : positive, square-free integer.
\item[] $\epsilon_{D}$: the fundamental unit of real quadratic field $\mathbb{Q}(\sqrt{D})$.
\item[] $\mathcal{N}$: the norm operator on $\mathbb{Q}(\sqrt{D}) / \mathbb{Q}$
\item[] $\ceil[big]{x}$: the smallest positive integer greater than or equals to some real number $x$
\end{itemize}
\section{Main Results}
\begin{thm}\label{thm1}
 Let $F$ be any number field in which $-1$ can be written as sum of integral squares. Then
$$
p(\mathcal{O}_{F}) \leq \begin{cases}
s(\mathcal{O}_{F})+1 & \text{ if } s(\mathcal{O}_{F}) \text{ is even},\\
s(\mathcal{O}_{F})+2 & \text{ if } s(\mathcal{O}_{F}) \text{ is odd.}
\end{cases}
$$
\end{thm}
This theorem is already proved by Peters(\cite{Pet}, Satz 1, Satz 3) in 1974. But here we have given an elementary proof of the result using Moser's theorem \cite{MO} (Proved in 1970).\\
 As an immediate consequence of this result one has the following:
\begin{cor}\label{cor1}
The Pythagoras number of ring of integers of a CM field is at most five.
\end{cor}

\begin{thm}\label{thm2}
 Let $ m \equiv  3 \pmod 4 $ and $ n \equiv 1 \pmod4 $ be two distinct square-free positive integers.
Assume that $K = \mathbb{Q}(\sqrt{-m},\sqrt{n})$. Then $ {R}_{K}  =  \mathcal{O}_{K}$.
\end{thm}

Applying these Theorems we get our final result:
\begin{thm}\label{thm3}
Let $K$ be any complex bi-quadratic field. Every element of $4 \mathcal{O}_K $ can be written as sum of five integral squares.
\end{thm}
But it is yet to prove that whether $4$ is smallest such integer for which this property holds.
Using Theorem \ref{thm2} and \ref{thm3} we get the following.

\begin{cor}\label{cor2}
Let $K$ be any complex bi-quadratic field.

\begin{itemize}
    \item[(i)] If $K=\mathbb{Q}(\sqrt{-m},\sqrt{n})$, where $m\equiv 3 \pmod 4$ and $n\equiv 1 \pmod 4 $, and if $s(\mathcal{O}_{K})=2$ then $p(\mathcal{O}_{K})=3$

   \item[ (ii) ] Otherwise, if  $s(\mathcal{O}_{K})=2$ and then $p(4\mathcal{O}_{K})=3$.

\end{itemize}

\end{cor}

 \section{Useful Results}

In this section we will summarize some results on quadratic and bi-quadratic fields which are beneficial to prove our main results.\\
In 1970, Moser \cite{MO} proved the following result which will be used in the proof of Corollary \ref{cor1}, Theorem \ref{thm2} and Theorem \ref{thm3}
\begin{thma}\label{thmM}
Let $K = \mathbb{Q}(\sqrt{-D})$ be an imaginary quadratic field. Then
$$
 s(K)=\begin{cases}
1 & \mbox{ if } D=1,\\
4 & \text{ if } D \equiv 7 \pmod 8, \\
2 & \mbox{ otherwise,}
\end{cases}
$$
Where $s(K)$ = the smallest number of squares
required for  representation of $-1$ as sum of squares of elements of $K$.
and 
$$
  s(\mathcal{O}_{K})=\begin{cases}
1 & \mbox{ if } D=1,\\
4 & \mbox{ if } D \equiv 7 \pmod 8, 
\end{cases}
$$
Otherwise,
$$ s(\mathcal{O}_{K})= \begin{cases}
2 & \mbox{ if } \mathcal{N}(\epsilon_{D}) = 1\\
3 &  \mbox{ if } \mathcal{N}(\epsilon _{D}) = -1
\end{cases}
$$
\end{thma}
In 1940, Ivan Niven \cite{Ni} proved the following two theorems on complex quadratic fields which we will use in the proof of Theorem \ref{thm3}. 
\begin{thma}\label{thmb}
Every integer of the form $a + 2b \sqrt{-D}$ of the complex quadratic field $\mathbb{Q}(\sqrt{-D})$ is expressible as a sum of three squares of integers of the field.
\end{thma}
In \cite{Wi} K.S William classified the integral basis for all bi-quadratic fields. For this, he showed that, without loss of generality we can assume $(m,n) \equiv (1,1), (1,2), (2,3), (3,3) \pmod 4$, where $m$ and $n$ are two distinct, square-free integers(not necessarily positive), and proved the following theorem.
\begin{thma}\label{thmc}
Let $(m,n)$ be as above, $gcd(m,n)=d$, $m_{1}=\frac{m}{d}$, $n_{1}=\frac{n}{d}$, $l= \frac{mn}{d^2}$ and then the integral basis of $\mathbb{Q}(\sqrt{m}, \sqrt{n})$ is given by
\begin{itemize}
    \item[(i)]$\{1, \frac{1+\sqrt{m}}{2}, \frac{1+\sqrt{n}}{2}, \frac{1+\sqrt{m}+\sqrt{n}+\sqrt{l}}{4}$\} if $m\equiv n \equiv 1 \pmod4$ and $m_{1} \equiv n_{1} \equiv 1 \pmod 4$.
    \item[(ii)]$\{1, \frac{1+\sqrt{m}}{2}, \frac{1+\sqrt{n}}{2}, \frac{1-\sqrt{m}+\sqrt{n}+\sqrt{l}}{4}$\} if $m\equiv n \equiv 1 \pmod4$ and $m_{1} \equiv n_{1} \equiv 3 \pmod 4$
    \item[(iii)]$\{1, \frac{1+\sqrt{m}}{2}, \sqrt{n}, \frac{\sqrt{n}+\sqrt{l}}{2}$\} if $m\equiv 1 \pmod4$ and $ n \equiv 2 \pmod 4$
    \item[(iv)]$\{1, \sqrt{m}, \sqrt{n}, \frac{\sqrt{m}+\sqrt{l}}{2}$\} if $m\equiv 2 \pmod4$ and $ n \equiv 3 \pmod 4$
    \item[(v)]$\{1, \sqrt{m}, \frac{\sqrt{m}+\sqrt{n}}{2}, \frac{1+\sqrt{l}}{2}$\} if $m\equiv n \equiv 3 \pmod4$ 
\end{itemize}
\end{thma}
In \cite{ZJ} Zang and Ji proved the following result on sum of integral squares over certain complex bi-quadratic fields, which we will use in Theorem \ref{thm2} and Theorem \ref{thm3}.
\begin{thma}\label{thmd}
Let $K= \mathbb{Q}(\sqrt{-m}, \sqrt{-n})$, where $m \equiv n \equiv 3 \pmod 4$ are two distinct, positive, square-free integers, then $\mathcal{O}_{K} = \mathcal{R}_{K}$ 
\end{thma}
Recently, Kala and Yatsyna \cite{VY} proved the following result for real quadratic fields which we will use in Theorem \ref{thm3}.
\begin{thma}\label{thme}
Let ${K}=\mathbb{Q}(\sqrt{D})$ with $ D \geq 2$ square-free. Let $\kappa = 1$ if $D \equiv 1 \pmod4 $ and $\kappa = 2$ if $ D \equiv 2, 3 \pmod 4 $.
\begin{itemize}
 \item[a)] If $ s < \kappa \frac{\sqrt{D}}{4}$ , then not all elements of $s\mathcal{O}_{K}^{+}$ are represented as the sum of squares in $\mathcal{O}_{K}$
 \item[b)]  If $ s \geq \frac{D}{2}$ , then all elements of $ \kappa  s\mathcal{O}^{+}$ are represented as the sum of five squares in $\mathcal{O}_{K}$
 \item[(c)] If $s$ is odd and $D \equiv 2,3 \pmod 4 $ then there exist elements of $s \mathcal{O}_{K}^{+}$ that are not sums of squares in $\mathcal{O}_{K}$
  \end{itemize}
\end{thma}

\section{Preparation}

In this section, using Theorem \ref{thmc}, and after making some possible changes, we will give a complete list of integral bases $\mathcal{B}$ for complex bi-quadratic fields and will prove some lemmas which are helpful to establish our main results.\\
An analogous result of the following was proved in \cite{ZJ} by Zhang and Ji and our proof goes along the similar line. We include a proof for the sake of completeness.

\begin{lem}\label{lem1}
Let $K = \mathbb{Q}(\sqrt{-m},\sqrt{n})$ be as in Theorem \ref{thm1} and  $\ell=-mn/\gcd(m,n)^2$. Then 
$$ 
\mathcal{B}: = \left\{1, \frac{(1+ \sqrt{-m})(1+ \sqrt{\ell})}{4}, \frac{1+\sqrt{-m}}{2},\frac{1+\sqrt{n}}{2} \right\}
$$
is an integral basis for $\mathcal{O}_{K}$.
\end{lem}

\begin{proof}
Let there exist $a_1, a_2, a_3, a_4\in \mathbb{Q}$, 
$$
a_1 +  a_2 \frac{(1+\sqrt{-m})(1+\sqrt{\ell})}{4}+a_3 \frac{1+\sqrt{-m}}{2}+ a_4\frac{1+\sqrt{n}}{2} = 0.
$$
Then $ a_4=0,  \frac{a_2(1+\sqrt{\ell})}{2} +a_3=0$ and  $a_1+\frac{a_2(1+\sqrt{\ell})}{4}+\frac{a_3}{2}=0$. These together imply that $a_1=a_2=a_3=a_4=0$, which shows that  
$\mathcal{B}$ is a basis for $K$ over $\mathbb{Q}$. Thus for any $\alpha \in \mathcal{O}_K$, we can write
$$ \alpha = x_1 +x_2 \frac{(1+\sqrt{-m})(1+\sqrt{\ell})}{4}+  x_3 \frac{1+\sqrt{-m}}{2}+ x_4 \frac{1+\sqrt{n}}{2}.$$ 
Thus it suffices to show that $x_i (i=1,\ldots,4)\in \mathbb{Z}$ to prove that $\mathcal{B}$ is indeed an integral basis for $\mathcal{O}_K$. 

Note that $\ell\equiv 1 \pmod 4$ and $K$ has three quadratic sub-fields ${K}_{1}= \mathbb{Q}(\sqrt{-m}), {K}_{2}= \mathbb{Q}(\sqrt{n})$ and ${K}_{3}= \mathbb{Q}(\sqrt{\ell})$. Now,
\begin{eqnarray*}
Tr_{K/{K}_{1}}(\alpha) &=& \left(2x_{1} + x_{4})+(x_{2} + 2x_{3}\right)\frac{1+\sqrt{-m}}{2} ;\\
Tr_{K/ {K}_{2}}(\alpha) &=& \left(2x_{1} + \frac{1-\frac{m}{(m,n)}}{2}x_{2} + x_{3}\right) + \left( \frac{m}{(m,n)}x_{2}+ 2x_{4}\right)\frac{1+\sqrt{n}}{2} ;\\
Tr_{K /{K}_{3}}(\alpha)&=&\left(2x_{1}+ x_{3}+x_{4}\right)+x_{2} \frac{1+\sqrt{\ell}}{2} .
\end{eqnarray*}
Since $Tr_{K/{K}_{3}}(\alpha) \in \mathcal{O}_{{K}_{3}} $, so that
$
2x_{1}+x_{3}+x_{4} \in \mathbb{Z}.
$
Similarly $Tr_{K/ {K}_{1}}(\alpha) \in \mathcal{O}_{{K}_{1}}$ gives that $
2x_{1}+x_{4} \in \mathbb{Z}$.
These together imply  $x_{3} \in \mathbb{Z}$.  

Again $Tr_{K /{K}_{2}}(\alpha) \in \mathcal{O}_{{K}_{2}}$ gives
$
2x_{1} + \left(\frac{1-\frac{m}{(m,n)}}{2}\right) x_{2}+ x_{3} \in \mathbb{Z}.
$
This further implies $2x_{1} \in  \mathbb{Z} $, and hence $ x_{4} \in \mathbb{Z}$ too.
Therefore, 
$$
x_1 = \alpha -x_2 \frac{(1+\sqrt{-m})(1+\sqrt{\ell})}{4}-  x_3 \frac{1+\sqrt{-m}}{2}- x_4 \frac{1+\sqrt{n}}{2},
$$
which implies that $x_{1} \in \mathcal{O}_{K} \cap  \mathbb{Q} = \mathbb{Z}$. 
This completes the proof. 
\end{proof}
\subsection*{Integral bases for complex bi-quadratic fields}\label{ib}
\begin{itemize}
    \item[(A)] Let $K= \mathbb{Q}(\sqrt{-m}, \sqrt{-n})$ where $m$ and $n$ are two distinct, positive, square-free integers, $gcd(m,n)=d$, and $l = \frac{mn}{d^2}$. Then,
    \begin{itemize}
        \item[(i)] \label{ib1} $\mathcal{B} =\{1, \frac{1+\sqrt{-m}}{2}, \frac{1+\sqrt{-n}}{2}, \frac{(1+\sqrt{-m})(1+\sqrt{l})}{4}$\} if $m\equiv n \equiv 3 \pmod4$
    \item[(ii)]$\mathcal{B} = \{1, \frac{1+\sqrt{-m}}{2}, \sqrt{-n}, \frac{\sqrt{-n}+\sqrt{l}}{2}$\} if $m\equiv 3 \pmod4$ and $ n \equiv 2 \pmod 4$
    \item[(iii)]$\mathcal{B} = \{1, \sqrt{-m}, \sqrt{-n}, \frac{\sqrt{-m}+\sqrt{l}}{2}$\} if $m\equiv 2 \pmod4$ and $ n \equiv 1 \pmod 4$
    \item[(iv)]$\mathcal{B} = \{1, \sqrt{-m}, \frac{\sqrt{-m}+\sqrt{-n}}{2}, \frac{1+\sqrt{l}}{2}$\} if $m\equiv n \equiv 1 \pmod4$ 
    \end{itemize}
    \item[(B)]  Let $K= \mathbb{Q}(\sqrt{-m}\sqrt{n})$ where $m$ and $n$ are two distinct, positive, square-free integers, $gcd(m,n)=d$, and $l = \frac{-mn}{d^2}$. Then,
    \begin{itemize}
        \item [(i)] \label{ib2} $\mathcal{B} = \{1, \frac{1+\sqrt{-m}}{2}, \frac{1+\sqrt{n}}{2}, \frac{(1+\sqrt{-m})(1+\sqrt{l})}{4}$\} if $m\equiv 3 \pmod4$, $n \equiv 1 \pmod 4$
    \item[(ii)]$\mathcal{B} = \{1, \frac{1+\sqrt{-m}}{2}, \sqrt{n}, \frac{\sqrt{n}+\sqrt{l}}{2}$\} if $m\equiv 3 \pmod4$ and $ n \equiv 2 \pmod 4$
    \item[(iii)]$\mathcal{B} = \{1, \sqrt{-m}, \sqrt{n}, \frac{\sqrt{-m}+\sqrt{l}}{2}$\} if $m\equiv 2 \pmod 4$ and $ n \equiv 3 \pmod 4$
    \item[(iv)]$\mathcal{B} = \{1, \sqrt{-m}, \frac{\sqrt{-m}+\sqrt{n}}{2}, \frac{1+\sqrt{l}}{2}$\} if $m\equiv 1 \pmod 4$, $n \equiv 3 \pmod 4$ 
    \end{itemize}
    \end{itemize}


\begin{lem}\label{loc}
Let $ \alpha \in 4 \mathcal{O}_F $ in the local field $F$ with the uniformizer $ \pi $. Then there is an integer $ \beta $ such that
\begin{equation*}
1+ \pi\alpha = (1+2\pi\beta)^{2}.
\end{equation*}
\end{lem} 
\begin{proof}
The proof directly follows from 'Local square Theorem' {\cite[See p. 159]{OM}}.
\end{proof}
\section{Proof of main theorems}
\begin{proof}[\bf Proof of Theorem \ref{thm1}]
Let $F$ be a number field in which $-1$ can be written as sum of integral squares. Let $\mathcal{O}_{F}$ be the ring of integers of $F$ and $\mathcal{R}_F$ be the set of all elements in $\mathcal{O}_F$ which are sums of squares in $\mathcal{O}_F$ and 
thus $\mathcal{R}_{F}$ is a sub-ring of $\mathcal{O}_{F}$.
By definition, $s(\mathcal{O}_{F}) \leq   p(\mathcal{O}_{F})$.
 
As $\alpha \in \mathcal{O}_{F}$ so do $ -\alpha$.
Hence there exists $\beta_{1},\beta_{2},\cdots,\beta_{t} \in \mathcal{O}_{F}$ such that
\begin{equation*}
-\alpha = \beta_{1}^{2}+\beta_{2}^{2}+...+\beta_{t}^{2}.
\end{equation*}
Then
\begin{eqnarray*}
\alpha &= &\left( \sum_{1\leq i \leq t}\beta_{i} + \sum_{1\leq i \leq j \leq t} \beta_{i}\beta_{j} +1\right)^{2} - \left( \sum_{1\leq i \leq t}\beta_{i} + \sum_{1\leq i \leq j \leq t} \beta_{i}\beta_{j}\right)^{2}\\
& -& \left(\sum_{1\leq i \leq t}\beta_{i}  + 1\right)^{2}.
\end{eqnarray*}
We write,
$\gamma_1=\sum\limits_{1\leq i \leq t}\beta_{i} + \sum\limits_{1\leq i \leq j \leq t} \beta_{i}\beta_{j} +1$, $\gamma_2=\sum\limits_{1\leq i \leq t}\beta_{i} + \sum\limits_{1\leq i \leq j \leq t} \beta_{i}\beta_{j}$ and $\gamma_3=\sum\limits_{1\leq i \leq t}\beta_{i}  + 1$. Then
$\gamma_{i} (i=1,2,3) 
\in \mathcal{O}_{F}$ which satisfy
\begin{equation*}
\alpha = \gamma_{1}^{2}+(-1)(\gamma_{2}^{2}+\gamma_{3}^{2}).
\end{equation*}
If $s(\mathcal{O}_{F}) = 2m$ for some positive integer $m$, then
\begin{eqnarray*}
\alpha &= &\gamma_{1}^{2}+\left(\sum_{i=1}^{2m}\epsilon_{i}^{2}\right) (\gamma_{2}^{2}+\gamma_{3}^{2})\\
&=&\gamma_{1}^{2} + \sum_{i=1}^{2m}\left( (\epsilon_{i}\gamma_{2}+\epsilon_{i+1}\gamma_{3})^{2} + (\epsilon_{i}\gamma_{3} - \epsilon_{i+1}\gamma_{2})^{2}\right),
\end{eqnarray*}
with $i$ varies over odd integers.
This shows that $p(\mathcal{O}_{F}) \leq 2m+1.$ 

Analogously, if $s(\mathcal{O}_{F}) = 2m + 1$ for some positive integer $m$, then
\begin{eqnarray*}
\alpha &=& \gamma_{1}^{2}+\left(\sum_{i=1}^{2m}\epsilon_{i}^{2}\right)(\gamma_{2}^{2}+\gamma_{3}^{2}) + \epsilon_{2m+1}^2 (\gamma_{2}^{2}+\gamma_{3}^{2})\\
&=&\gamma_{1}^{2} + \sum_{i=1}^{2m}\left(\epsilon_{i}\gamma_{2}+\epsilon_{i+1}\gamma_{3})^{2} + (\epsilon_{i}\gamma_{3} - \epsilon_{i+1}\gamma_{2})^{2}\right) + \epsilon_{2m+1}^2(\gamma_{2}^{2}+\gamma_{3}^{2}),
\end{eqnarray*}
with $i$ varies over odd integers. This shows that $p(\mathcal{O}_{F})\leq 2m+3 $.\\
This completes the proof.
\end{proof}
\begin{proof}[\bf Proof of Corollary \ref{cor1}]
A CM field $F$ is a complex quadratic extension over totally real number field. Therefore it has a complex quadratic sub-field, say, $F_{1}$. From Theorem \ref{thmM} we get $-1$ can be written as sum of $2, 3$ or $4$ integer squares in $F_{1}$. Then it can also be written as sum of at most four integer squares in $F$. Then the result follows immediately from Theorem \ref{thm1}.
\end{proof}
In fact, this result is true for any number field which has at least one complex quadratic sub-field.
\begin{proof}[\bf Proof of Theorem \ref{thm2}]
By definition, $\mathcal{R}_{K} \subseteq \mathcal{O}_{K}$.
Therefore it is sufficient to show that $\mathcal{O}_{K} \subseteq \mathcal{R}_{K}$ to complete the proof.
By the four square theorem, $\mathbb{N}\subseteq \mathcal{R}_{K}$. Also $-1\in \mathcal{R}_K$ by Theorem \ref{thmM} and hence $\mathbb{Z}\subseteq \mathcal{R}_{K}$. 
Since $\mathcal{B}$ (in Proposition \ref{lem1}) is an integral basis for $\mathcal{O}_K$, any $\alpha\in \mathcal{O}_K$ can be expressed as
\begin{equation}\label{eqal}
\alpha = x_1 + x_2 \left(\frac{(1+\sqrt{-m})(1+\sqrt{ \ell})}{4}\right)+x_{3} \left(\frac{1+\sqrt{-m}}{2}\right)+ x_{4} \left(\frac{1+\sqrt{n}}{2}\right),
\end{equation}
where $ x_{i}
\in \mathbb{Z}, i= 1, \cdots, 4$ and $\ell= \frac{mn}{\gcd(m,n)^{2}}$. 

Now
\begin{align}\label{eqx}
\begin{cases}
 \dfrac{1+\sqrt{-m}}{2} = \left(\dfrac{1+\sqrt{-m}}{2}\right)^{2} + \dfrac{m+1}{4} \\
\dfrac{1+\sqrt{n}}{2} = \left(\dfrac{1+\sqrt{n}}{2}\right)^{2} -\dfrac{n-1}{4} \\
 \dfrac{1+\sqrt{\ell}}{2} = \left(\dfrac{1+\sqrt{\ell}}{2}\right)^{2} - \dfrac{\ell-1}{4} .
 \end{cases}
\end{align}
Since $m \equiv 3 \pmod 4 $ and $ n, \ell \equiv 1 \pmod 4$, one has  $\frac{m+1}{4}, \frac{n-1}{4}, \frac{\ell-1}{4} \in \mathbb{Z}$, and thus $\frac{m+1}{4}, \frac{n-1}{4}, \frac{\ell-1}{4}\in\mathcal{R}_K$. Therefore using \eqref{eqx}, 
$$
 \frac{(1+ \sqrt{-m})(1+ \sqrt{\ell})}{4}, \frac{1+\sqrt{-m}}{2},\frac{1+\sqrt{n}}{2} \in \mathcal{R}_{K}.
$$
 The proof is now completed by \eqref{eqal}.
 \end{proof}
 
 Now, Using Theorem \ref{thm1} and Theorem \ref{thm2} we are ready to proof our main theorem.

 \begin{proof}[\bf Proof of Theorem \ref{thm3}]

We first prove that every element of $4\mathcal{O}_K$ for any complex bi-quadratic field $K$ can be written as sum of integral squares.\\
To prove this we first prove that for each complex bi-quadratic field $K$ there exist a positive integer $s_{0}$ such that for any $s \geq s_0$ every element of $4s\mathcal{O}_{K}$ can be written as sum of integral squares. By the result of Zang and Ji \cite{ZJ} and by Theorem \ref{thm2} we know that, for complex quadratic fields $\mathbb{Q}(\sqrt{-m}, \sqrt{-n})$, where $m \equiv n \equiv 3 \pmod 4 $ and for $\mathbb{Q}(\sqrt{-m}, \sqrt{n})$, where $ m \equiv 3 \pmod 4 $, $n \equiv 1 \pmod 4 $ every algebraic integer can be written as sums of integral squares. For other cases we can see from Theorem \ref{thmc} that every algebraic integer $\alpha \in \mathcal{O}_{K}$ can be written as,
\begin{equation}\label{eqn2}
    \alpha = a + b \sqrt{m} + c\sqrt{n} + d \sqrt{l}
\end{equation}
where, $m$ and $n$ are distinct, square-free integers(not always positive),  $l= \frac{mn}{gcd(m,n)^2}$ and $a,b,c,d$ are either integers or half-integers.\\

Then $2a,2b,2c,2d$ are always integers. From Theorem \ref{thmM} and by Lagrange's four square theorem we get every element of $\mathbb{Z}$ can be written as sum of integral squares in $K$. Therefore, for any half-integer $a$, $2a \in \mathcal{R}_{K}$. Again, by Theorem \ref{thmb} we know that $2\sqrt{-r}$ can be written as sum of integral squares in the sub-field $\mathbb{Q}(\sqrt{-r}$ of $K$, for any positive, square-free integer r, and hence in $K$. We also know that every complex bi-quadratic field $K$ has exactly one real quadratic sub-field. Let $K_{1} = \mathbb{Q}(\sqrt{D})$ be that sub-field. Then, by Theorem \ref{thme}, for $s \geq \ceil[big]{\frac{D}{2}}$, $\kappa s(\sqrt{D})$ can be written as sum of integral squares. Since $\kappa = 1$ or $2$, this implies  $s \sqrt{D}$ or $2s \sqrt{D}$ can be written as sum of integral squares.  Also, if $s(\sqrt{D})$ can be written as sum of squares in $\mathcal{O}_{K_1}$, $2(s(\sqrt{D}))$ also have the same property ( as $\mathbb{Z} \subseteq \mathcal{R}_{K}$ ) in $\mathcal{O}_{K}$.\\
Combining all these facts we get,\\
\begin{equation}\label{eqn3}
4as\sqrt{-r} = (2a)(s)(2\sqrt{-r})
\end{equation}
where each of $2a$, $s$ and $2\sqrt{-r}$ can be written as sum of squares in $\mathcal{O}_{K}$. \\
Similarly, we can say,
\begin{equation}\label{eqn4}
4as\sqrt{D} = (2a)(2s\sqrt{D})
\end{equation}
where each of $(2a)$ and $(2s\sqrt{D})$ can be written as sum of squares in $\mathcal{O}_{K}$. \\
Using equation \ref{eqn3} and \ref{eqn4} in equation \ref{eqn2} the claim follows.\\
Again, we can write
\begin{equation*}
 4\sqrt{D} = (1+\sqrt{D})^2 - (1- \sqrt{D})^2
\end{equation*}
Since $-1$ can be written as sum of integral squares in $K$, therefore $4\sqrt{D}$ can also be written as the same in $K$. This implies,
\begin{equation} \label{eqn6}
    8a\sqrt{D} = (2a)(4\sqrt{D})
\end{equation}
can also be written as the same.
Again, since $4a$ and $2\sqrt{-r}$ can be written as sum of integral squares in $\mathcal{O}_{K}$ we can say,
\begin{equation}\label{eqn7}
    8a \sqrt{-r}= (4a)(2 \sqrt{-r})
\end{equation}
can be written as the same in $\mathcal{O}_{K}$\\ 
Using equation \ref{eqn6} and \ref{eqn7} in equation \ref{eqn3} we get that for any positive integer $t$ every element of $8t\mathcal{O}_{K}$ can be written as sum of integral squares.
Now, since $D$ is a square-free positive integer, it must be of the form $D \equiv 2 \pmod 4 $. This implies $s_0$ is an odd integer.\\
Let $s_0 = 2n_0+1 $ for some positive integer $n_0$ and for some complex bi-quadratic field $K_0$. Then for any $ \alpha \in \mathcal{O}_{K_{0}} $ we can write 
\begin{equation}\label{eqn8}
    4(2n_0+1)\alpha = \sum\limits_{1\leq i \leq t_0}\alpha_{i}^2
\end{equation}
And again,
\begin{equation}\label{eqn9}
    8n_0 \alpha = \sum\limits_{1\leq i \leq t_1}\beta_{i}^2
\end{equation}
Subtracting equation \ref{eqn9} from \ref{eqn8} we get 
\begin{equation}
    4\alpha = \sum\limits_{1\leq i \leq t_0}\alpha_{i}^2 - \sum\limits_{1\leq i \leq t_1}\beta_{i}^2
\end{equation}
Since $-1$ can be written as sum of integral squares in $K_0$ and since $\alpha$ is an arbitrary  element of $K_0$ we can say that every element of $4\mathcal{O}_{K_{0}}$ can be written as sum of integral squares. Also, since this is true for each $K_0$, the claim follows for all complex bi-quadratic fields. 
Now, giving exactly similar argument in the proof of Theorem \ref{thm1} we conclude that every element of $4\mathcal{O}_{K}$ can be written as sum of five integral squares. This completes the theorem. 
\end{proof}

 \begin{proof}[\bf Proof of Corollary \ref{cor2} ]
 In \cite{ZJ}, Zang and Ji proved that, for $K=\mathbb{Q} (\sqrt{-m}, \sqrt{-n})$, where $ m\equiv n \equiv 3 \pmod 4 $, this result is true.
  In exactly similar way, for $K= \mathbb{Q}(\sqrt{-m}, \sqrt{n})$, where $m \equiv 3 \pmod 4$ and $n \equiv 1 \pmod4 $ we get the same result.
  We are giving the details for part (ii).\\
  Theorem \ref{thm2} gives that $ \mathcal{R}_{K} = \mathcal{O}_{K}$. 
Assume that $s( \mathcal{O}_{K}) = 2$. Then by Theorem \ref{thm1}, we  conclude that every element of $4\mathcal{O}_{K} $ can be expressed as a sum of three squares. Thus it remains only to show that  there exists an element in $4\mathcal{O}_{K}$ which is not a sum of two integral squares.

Let $L = K (\sqrt{-1})$.  
Assume that $\mathcal{P}$ is a prime ideal above $2$ in $K$ and $ \mathcal{Q }$ is a prime ideal above $\mathcal{ P}$ in $L$. Then $\mathcal{ P}$ is totally ramified in $L$. Let $ L_{\mathcal{Q}}$ and $K_{\mathcal{P}} $ denote the completions of $L$ and $K$ at $\mathcal{Q}$ and $\mathcal{P}$ respectively. Then [$L_{\mathcal{Q}}: K_{\mathcal{P}} ] = 2$ and by the local class field theory, ${K}_{\mathcal{P}}^{*}/{N}(L_{\mathcal{Q}}^{*}) \cong  \text{Gal}( {L}_{\mathcal{Q}} / {K}_{\mathcal{P}} )$. Thus [${K}_{\mathcal{P}}^{*} : {N}(L_{\mathcal{Q}}^{*} ]= 2$. Assume that every element in $ 4\mathcal{O}_{{K}}$ is a sum of two integral squares in ${K}$. However, ${K}$ is dense in ${K}_{\mathcal{P}} $,  thus by Lemma \ref{loc}, we get that every element of ${K}_{\mathcal{P}} $ is a sum of two squares in ${K}_{\mathcal{P}} $, that is , the norm of ${L}_{\mathcal{Q}}^{*} \rightarrow {K}_{\mathcal{P}}^{*}$ is surjective, which is a contradiction. Therefore $p(4\mathcal{O}_{{K}})=3$.
\end{proof}

\section{Examples}
First we are giving a table of fundamental units $\epsilon_D$ of real quadratic field $\mathbb{Q}(\sqrt{D})$ for which $\mathcal{N}(\epsilon_D) = 1$

\begin{center}
\begin{tabular}{ |c|c| } 
 \hline
 $D$ & $\epsilon_{D}$ \\ 
 \hline
 $3$ & $2+\sqrt{3}$ \\ 
 $6$ & $5+2\sqrt{6}$\\ 
 $7$ & $8+3\sqrt{7}$\\
 $11$ & $10+3\sqrt{11}$\\
 $14$ & $15+4\sqrt{14}$\\
 $15$ & $4+\sqrt{15}$\\
 $21$ & $\frac{1}{2}(5+\sqrt{21})$\\
 \hline
\end{tabular}
\end{center}
From the table and using Theorem \ref{thmM} we can say, if $K= \mathbb{Q}(\sqrt{-m},  \sqrt{-n})$, or $\mathbb{Q}(\sqrt{-m}, \sqrt{-n})$ where $m=3,6,11,14$ or $21$ and $n$ is any positive, square-free integer, $s(\mathcal{O}_{K})=2$\\
And if $K= \mathbb{Q}(\sqrt{-m}, \sqrt{-n})$, where $m=7,15$ and $n \neq m $ is any positive, square-free integer, then also $s(\mathcal{O}_{K})=2$.\\
In general, if a bi-quadratic field $K$ has a sub-field $\mathbb{Q}(\sqrt{-D})$ which has an odd class number and $ D= p,2p, p_{1},p_{2}$ with $p\equiv p_{1} \equiv p_{2} \equiv 3 \pmod4 $ then from \cite{PJ} we get $s(\mathcal{O}_{K})=2$.\\
 
\section*{Acknowledgments}
The author expresses her gratitude to her adviser  Prof.  Kalyan Chakraborty  for going through the manuscript and revising it thoroughly. The author is also indebted  to Dr. Azizul Hoque for introducing her into this beautiful area of research, and for many fruitful comments and valuable suggestions.

\end{document}